\documentclass[11pt]{amsart}

\usepackage[T1]{fontenc}
\usepackage[utf8]{inputenc}
\usepackage{amsmath,amssymb,amsfonts,amsthm,amscd}
\usepackage{graphicx,url}
\usepackage{indentfirst}

\usepackage{geometry}
\allowdisplaybreaks
\theoremstyle{plain}
\newtheorem{thm}{Theorem}[section]
\newtheorem{lem}[thm]{Lemma}

\newtheorem{prop}[thm]{Proposition}

\theoremstyle{definition}

\newtheorem{defn}[thm]{Definition}

\begin{document}

\title[Rational injectivity in equivariant $KK$-theory]
	{\uppercase{R}\lowercase{ational injectivity
			in equivariant }$KK$\lowercase{-theory}}

	\author{Denghui Luo}
	\address{Department of Mathematics, Jilin University, Changchun,  China} 
	\email{denghuiluo@foxmail.com}
	
\thanks{}
	
	\author{Geng Tian}
	
	\address{Department of Mathematics, Liaoning University, Liaoning,  China}
	\email{gengtian.ncg@gmail.com}
	
\thanks{}

\renewcommand{\thefootnote}{}

	\begin{abstract}
		Antonini, Azzali and Skandalis proved that, for every countable discrete group $\Gamma$, the canonical map $E\Gamma\to\underline{E}\Gamma$ induces an injection in equivariant $KK$-theory with real coefficients.
		In this paper, we prove the corresponding rational injectivity result in  equivariant $KK$-theory without requiring real coefficients.
	\end{abstract}
	
	\maketitle

	\section{Introduction}

	Throughout, $\Gamma$ is a countable discrete group.
	A $\Gamma$-$C^*$-algebra is a $C^*$-algebra equipped
	with an action of $\Gamma$ by $*$-automorphisms.
	An extremely potent tool in noncommutative geometry, particularly in relation with the Novikov conjecture, is Kasparov's equivariant $KK$-theory \cite{Kas88}.
	For   $\Gamma$-$C^*$-algebras $A$ and $B$,
	Kasparov's equivariant $KK$-theory 
	defines an abelian group $KK^\Gamma(A,B)$ consisting
	of homotopy classes of equivariant Kasparov
	modules.
	It is contravariant in $A$ and covariant in $B$
	with respect to equivariant $*$-homomorphisms,
	and is homotopy invariant, stable and split exact.
	The theory  satisfies Bott periodicity:
	\[
	KK^\Gamma(A,B)
	\cong KK^\Gamma(\Sigma^2 A,B)
	\cong KK^\Gamma(\Sigma A,\Sigma B)
	\cong KK^\Gamma(A,\Sigma^2 B),
	\]
	where $\Sigma^k A=C_0(\mathbb{R}^k,A)$ and
	$\Gamma$ acts trivially on $\mathbb{R}^k$. For $n\in\mathbb{Z}$, we write
	\[
	KK_n^\Gamma(A,B)
	:=
	KK^\Gamma(A,\Sigma^k B),
	\qquad k\in\{0,1\},\quad k\equiv n\pmod 2.
	\]
	
	The Kasparov product is an associative bilinear pairing
	\[
	KK^\Gamma(A,B)\otimes_{\mathbb{Z}}KK^\Gamma(B,C)
	\longrightarrow KK^\Gamma(A,C),
	\qquad
	x\otimes y\longmapsto x\otimes_B y.
	\]
	An equivariant $*$-homomorphism $\varphi:A\to B$
	determines a class $[\varphi]\in KK^\Gamma(A,B)$.
	When $\Gamma$ is trivial, we write $KK(A,B)$.
	For a general $\Gamma$, there is
	a natural forgetful homomorphism
	$KK^\Gamma(A,B)\to KK(A,B)$.

	A subset $S$ of a $\Gamma$-space $X$ is called
	$\Gamma$-compact if $S\subseteq\Gamma\cdot K$ for some compact
	subset $K\subseteq X$.
	A $\Gamma$-CW-complex is called finite if it has finitely
	many $\Gamma$-orbits of cells, and proper if all its
	stabilizer subgroups are finite;
	see \cite[Section~1]{LOCompletion01}.
	
	\begin{defn}\label{defn:KK-Gam-compact}
		Let $X$ be a Hausdorff $\Gamma$-space, let $B$ be
		a $\Gamma$-$C^*$-algebra, and let $n\in\mathbb{Z}$.
		Choose $k\in\{0,1\}$ with $k\equiv n\pmod 2$, and define
		\[
		KK_n^\Gamma(X,B)
		=
		\varinjlim_{Y,A}
		KK^\Gamma\bigl(C_0(Y),\Sigma^k A\bigr),
		\]
		where $Y$ ranges over the closed, $\Gamma$-invariant,
		$\Gamma$-compact subspaces of $X$ that are locally compact
		and second countable, and $A$ ranges over the separable
		$\Gamma$-invariant $C^*$-subalgebras of $B$,
		both ordered by inclusion.
	\end{defn}
	
	Let $E\Gamma$ be the universal space for proper and free actions of the group $\Gamma$, and let $\underline{E}\Gamma$ be the universal space for proper actions of the group $\Gamma$. 
	Then there exists a natural $\Gamma$-equivariant continuous map $\pi: E\Gamma \to \underline{E}\Gamma$.
	Antonini, Azzali and Skandalis
	\cite[Theorem~5.4]{AAS20} proved that the induced map
	\[
	\pi_{*,\mathbb{R}}:
	KK^\Gamma_{*,\mathbb{R}}(E\Gamma,A)
	\longrightarrow
	KK^\Gamma_{*,\mathbb{R}}(\underline E\Gamma,A)
	\]
	is split injective, where the groups are defined using
	equivariant $KK$-theory with real coefficients.
	In this paper, we remove the real coefficients and prove rational injectivity in equivariant $KK$-theory.
	We state the main result as follows.
	
	\begin{thm}\label{thm:main}
		Let $\Gamma$ be a countable discrete group.
		For every $\Gamma$-$C^*$-algebra $A$,
		the homomorphism
		\[
		\pi_*:KK_*^\Gamma(E\Gamma,A)
		\longrightarrow KK_*^\Gamma(\underline E\Gamma,A)
		\]
		is rationally injective,
		where $\pi_*$ is induced by the 
		natural $\Gamma$-equivariant continuous map $\pi: E\Gamma \to \underline{E}\Gamma$.
	\end{thm}
	
	We choose countable $\Gamma$-CW-models for $E\Gamma$ and
	$\underline E\Gamma$.
	Since every $\Gamma$-compact subset of either model
	is contained in a finite $\Gamma$-subcomplex,
	the limit over $Y$ in
	Definition~\ref{defn:KK-Gam-compact} may be taken over
	the finite $\Gamma$-subcomplexes of the chosen model.

	\emph{AI disclosure.}  ChatGPT was used solely for language editing and grammatical
	refinement during the preparation of this manuscript. All mathematical proofs and
	core results were developed by the authors.

	\section{Preliminaries}
	
	We recall exact sequences in the first variable of equivariant
	$KK$-theory and the action of equivariant vector bundles,
	including its compatibility with boundary maps.
	We then construct rational idempotents using the equivariant
	Chern character.
	For an abelian group $M$, we write
	$M_{\mathbb{Q}}=M\otimes_{\mathbb{Z}}\mathbb{Q}$.
	If an idempotent $e$ acts on $M$, we write
	$eM=\operatorname{Im}(e:M\to M)$.
	
	\subsection{Exact sequences of $KK$-theory in the first variable}
	\label{subsec:extensions}
	
	For nuclear proper $\Gamma$-$C^*$-algebras in the first
	variable, \cite[Corollary~A.4]{KS2003} identifies
	equivariant $KK$-theory with equivariant $E$-theory.
	Under this identification,
	\cite[Proposition~5.8 and Theorem~6.18]{GHT00}
	give the following statement.
	
	\begin{prop}[\cite{GHT00}]\label{prop:extension}
		Let $\Gamma$ be a countable discrete group, and let
		\begin{equation}\label{eq:extension}
			0\longrightarrow J\xrightarrow{j}D\xrightarrow{q}Q
			\longrightarrow0
		\end{equation}
		be an equivariant extension of separable nuclear proper
		$\Gamma$-$C^*$-algebras.
		There is a natural boundary  class
		$\delta_q\in KK_1^\Gamma(Q,J)$ such that, for every separable
		$\Gamma$-$C^*$-algebra $B$, the sequence
		\[
		\begin{aligned}
			\cdots&\longrightarrow KK_n^\Gamma(Q,B)
			\xrightarrow{q^*}KK_n^\Gamma(D,B)
			\xrightarrow{j^*}KK_n^\Gamma(J,B)
			\xrightarrow{\partial}KK_{n-1}^\Gamma(Q,B)
			\longrightarrow\cdots
		\end{aligned}
		\]
		is exact, where the boundary map $\partial$ is determined by
		\begin{equation}\label{eq:boundary-product}
			\partial(x)=\delta_q\otimes_J x.
		\end{equation}
		For a morphism between two such extensions, with induced maps
		$\phi:J\to J'$ and $\psi:Q\to Q'$, naturality means that
		\begin{equation}\label{eq:extension-naturality}
			\delta_q\otimes_J[\phi]
			=
			[\psi]\otimes_{Q'}\delta_{q'}.
		\end{equation}
	\end{prop}

	\subsection{Vector bundles and $KK_0^\Gamma(C_0(X),C_0(X))$}
	Throughout this subsection, all $\Gamma$-spaces are assumed
	to be second countable, locally compact and Hausdorff.
	We recall some standard facts from
	\cite[Section~3.4]{Raven04}; see also Kasparov's
	original paper \cite{Kas88}.
	
	For a proper $\Gamma$-space $X$, let $K_\Gamma^0(X)$ be the
	Grothendieck group associated to the monoid of isomorphism
	classes of finite-dimensional complex $\Gamma$-vector bundles
	over $X$.
	Tensor product of $\Gamma$-vector bundles gives $K_\Gamma^0(X)$
	the structure of a commutative ring.
	
	Let $V\to X$ be a $\Gamma$-equivariant complex vector bundle,
	equipped with a $\Gamma$-invariant Hermitian metric.
	The space $\Gamma_0(V)$ of continuous sections vanishing at
	infinity is a Hilbert $C_0(X)$-module, and multiplication defines
	a $\Gamma$-equivariant homomorphism
	\[
	\rho:C_0(X)\longrightarrow \mathcal{K}\big(\Gamma_0(V)\big),
	\]
	where $\mathcal{K}\big(\Gamma_0(V)\big)$ denotes the $C^*$-algebra of all compact operators on the
	Hilbert module $\Gamma_0(V)$.
	With the trivial grading and $F=0$, this gives an equivariant
	Kasparov module in the sense of
	\cite[Definition~2.2]{Kas88}, and hence a class
	\[
	\widehat{[V]}
	:=
	[(\Gamma_0(V),\rho,0)]
	\in KK_0^\Gamma\big(C_0(X),C_0(X)\big),
	\]
	where $[V]$ denotes the class of the vector bundle $V$ in $K_\Gamma^0(X)$.
	We now have
	a unital ring homomorphism
	\[
	K_\Gamma^0(X)
	\longrightarrow KK_0^\Gamma\big(C_0(X),C_0(X)\big),
	\qquad
	\xi=[V]-[W]\longmapsto
	\widehat{\xi}:=\widehat{[V]}-\widehat{[W]},
	\]
	where multiplication on the target is the Kasparov product.
	
	For every $\Gamma$-$C^*$-algebra $B$, we define
	\begin{equation}\label{eq:action}
		\xi\cdot x
		:=
		\widehat{\xi}\otimes_{C_0(X)}x,
		\qquad
		\xi\in K_\Gamma^0(X),\quad
		x\in KK_n^\Gamma\big(C_0(X),B\big).
	\end{equation}
	By the bilinearity and associativity of the Kasparov product,
	this makes $KK_n^\Gamma(C_0(X),B)$ a unital
	$K_\Gamma^0(X)$-module.
	Extending scalars gives an action of
	$K_\Gamma^0(X)_{\mathbb{Q}}$ on
	$KK_n^\Gamma(C_0(X),B)_{\mathbb{Q}}$.

	\begin{prop}[{\cite[Lemma~3.4.4(a) and
			Proposition~3.4.5]{Raven04}}]\label{prop:action}
		Let $\Gamma$ be a countable discrete group, $X$ a proper
		$\Gamma$-space, $B$ a separable $\Gamma$-$C^*$-algebra,
		and $n\in\mathbb{Z}$.
		Then the following statements hold.
		\begin{enumerate}
			\item
			Let $Y$ be another  proper $\Gamma$-space, and let $f:Y\to X$
			be a proper $\Gamma$-equivariant continuous map.
			For every $\xi\in K_\Gamma^0(X)$ and $x\in KK_n^\Gamma(C_0(Y),B)$, one has
			\begin{equation*}\label{eq:projection}
				f_*(f^*\xi\cdot x)=\xi\cdot f_*x,
			\end{equation*}
			where $f^*\xi$ is defined by pullback of vector bundles,
			and
			\[
			f_*:KK_n^\Gamma\big(C_0(Y),B\big)
			\longrightarrow KK_n^\Gamma\big(C_0(X),B\big)
			\]
			is induced by the pullback homomorphism
			$C_0(X)\to C_0(Y)$, $h\mapsto h\circ f$.
			
			\item
			Let $Z\subseteq X$ be a closed $\Gamma$-invariant subspace,
			and set $U=X\setminus Z$.
			Let
			\[
			\partial:KK_n^\Gamma(C_0(U),B)
			\longrightarrow KK_{n-1}^\Gamma(C_0(Z),B)
			\]
			be the boundary map given by \eqref{eq:boundary-product}.
			For every $\xi\in K_\Gamma^0(X)$ and
			$x\in KK_n^\Gamma(C_0(U),B)$,  one has
			\begin{equation*}\label{eq:boundary-module}
				\partial(\xi|_U\cdot x)
				=
				\xi|_Z\cdot\partial(x),
			\end{equation*}
			where $\xi|_U$ and $\xi|_Z$ denote the restrictions of
			$\xi$ to $U$ and $Z$, respectively.
		\end{enumerate}
		Both identities extend $\mathbb{Q}$-bilinearly to
		$\xi\in K_\Gamma^0(X)_{\mathbb{Q}}$ and the corresponding
		rationalized $KK$-groups.
	\end{prop}
	\begin{proof}
		For~(1), let
		$[f^\#]\in KK_0^\Gamma(C_0(X),C_0(Y))$
		be the class of the pullback homomorphism
		\[
		f^\#:C_0(X)\longrightarrow C_0(Y),
		\qquad f^\#(h)=h\circ f.
		\]
		By \cite[Lemma~3.4.4(a)]{Raven04}, 
		\[
		\widehat{\xi}\otimes_{C_0(X)}[f^\#]
		=
		[f^\#]\otimes_{C_0(Y)}\widehat{f^*\xi}.
		\]
		Since $f_*x=[f^\#]\otimes_{C_0(Y)}x$, it follows from \eqref{eq:action} and the associativity
		of the Kasparov product that
		\[
		\begin{aligned}
			f_*(f^*\xi\cdot x)
			&=
			[f^\#]\otimes_{C_0(Y)}
			\bigl(\widehat{f^*\xi}\otimes_{C_0(Y)}x\bigr)\\
			&=
			\bigl([f^\#]\otimes_{C_0(Y)}\widehat{f^*\xi}\bigr)
			\otimes_{C_0(Y)}x\\
			&=
			\bigl(\widehat{\xi}\otimes_{C_0(X)}[f^\#]\bigr)
			\otimes_{C_0(Y)}x\\
			&=
			\widehat{\xi}\otimes_{C_0(X)}
			\bigl([f^\#]\otimes_{C_0(Y)}x\bigr)\\
			&=\xi\cdot f_*x.
		\end{aligned}
		\]
		
		For~(2), let $q:C_0(X)\to C_0(Z)$ be the restriction
		homomorphism. Its boundary class is
		\[
		\delta_q\in KK_1^\Gamma(C_0(Z),C_0(U)).
		\]
		By \cite[Proposition~3.4.5]{Raven04}, 
		\[
		\widehat{\xi|_Z}\otimes_{C_0(Z)}\delta_q
		=
		\delta_q\otimes_{C_0(U)}\widehat{\xi|_U}.
		\]
		Using \eqref{eq:boundary-product}, \eqref{eq:action}  and associativity,
		we obtain
		\[
		\begin{aligned}
			\partial(\xi|_U\cdot x)
			&=
			\delta_q\otimes_{C_0(U)}
			\bigl(\widehat{\xi|_U}\otimes_{C_0(U)}x\bigr)\\
			&=
			\bigl(\delta_q\otimes_{C_0(U)}\widehat{\xi|_U}\bigr)
			\otimes_{C_0(U)}x\\
			&=
			\bigl(\widehat{\xi|_Z}\otimes_{C_0(Z)}\delta_q\bigr)
			\otimes_{C_0(U)}x\\
			&=
			\widehat{\xi|_Z}\otimes_{C_0(Z)}
			\bigl(\delta_q\otimes_{C_0(U)}x\bigr)\\
			&=\xi|_Z\cdot\partial(x).
		\end{aligned}
		\]
		Finally, both identities extend $\mathbb{Q}$-bilinearly
		to the corresponding rationalized groups.
	\end{proof}
	
	\subsection{The idempotents supplied by the Chern character}
	
	L\"uck and Oliver constructed, for every proper
	$\Gamma$-CW-complex $X$, a natural equivariant Chern character
	\[
	\mathrm{ch}_X^*:
	K_\Gamma^*(X)
	\longrightarrow
	H_\Gamma^*(X;\mathbb{Q}\otimes R(-));
	\]
	see \cite[Section~5]{LO01}.
	Here the target is Bredon cohomology, and $R(-)$ is the
	coefficient system assigning the complex representation ring
	$R(H)$ to $\Gamma/H$ for each finite subgroup $H\leq\Gamma$,
	with structure maps induced by restriction and conjugation.
	The Chern character sends $K_\Gamma^0(X)$ to even-degree
	cohomology and $K_\Gamma^1(X)$ to odd-degree cohomology.
	It preserves products and commutes with boundary maps in
	Mayer--Vietoris sequences.
	For finite proper $\Gamma$-CW-complexes, it becomes
	an isomorphism of rings after rationalization.
	
	\begin{prop}[{\cite[Theorem~5.5]{LO01}}]\label{fact:chern}
		For any finite proper $\Gamma$-CW-complex $X$, the equivariant Chern
		character extends to an isomorphism of rings
		\[
		\mathrm{ch}_X^*\otimes\mathbb{Q}:
		K_\Gamma^*(X)_\mathbb{Q}
		\overset{\cong}{\longrightarrow}
		H_\Gamma^*(X;\mathbb{Q}\otimes R(-)).
		\]
		This isomorphism is natural in $X$.
	\end{prop}

	For each finite subgroup $H\leq\Gamma$, set
	\[
	\varepsilon_H=\frac{[\mathbb{C}H]}{|H|}
	\in R(H)_{\mathbb{Q}},
	\]
	where $\mathbb{C}H$ is the complex left regular representation of $H$.
	It is an idempotent, i.e., $\varepsilon_H^2=\varepsilon_H$.
	The following lemma is an immediate consequence of
	\cite[Theorem~5.5]{LO01}.
	\begin{lem}\label{lem:idempotent}
		For each finite proper $\Gamma$-CW-complex $X$, there is an
		idempotent $e_X\in K_\Gamma^0(X)_{\mathbb{Q}}$ with the
		following properties:
		\begin{enumerate}
			\item
			Under $K_\Gamma^0(\Gamma/H)_{\mathbb{Q}}
			\cong R(H)_{\mathbb{Q}}$,
			\[
			e_{\Gamma/H}=\varepsilon_H.
			\]
			
			\item
			For every equivariant map $f:X\to X'$ between finite proper
			$\Gamma$-CW-complexes,
			\[
			f^*e_{X'}=e_X.
			\]
			
			\item
			If the $\Gamma$-action on $X$ is free, then
			\[
			e_X=1,
			\]
			where $1$ is the multiplicative unit of
			$K_\Gamma^0(X)_{\mathbb{Q}}$.
		\end{enumerate}
	\end{lem}
	\begin{proof}
		The idempotents $\varepsilon_H$ are compatible with
		restriction and conjugation, and hence define a natural
		idempotent
		\[
		c_X\in H_\Gamma^0(X;\mathbb{Q}\otimes R(-)).
		\]
		By Proposition \ref{fact:chern}, there is a unique
		$e_X\in K_\Gamma^0(X)_{\mathbb{Q}}$ such that
		\[
		\big(
		\mathrm{ch}_X^0\otimes\mathbb{Q}
		\big)(e_X)
		=(c_X,0,0,\ldots),
		\]
		where all positive-degree components are zero.
		Since the rational Chern character is a ring isomorphism, 
		$c_X^2=c_X$ implies $e_X^2=e_X$.
		Its naturality, together with $f^*c_{X'}=c_X$, gives
		$f^*e_{X'}=e_X$.
		
		For $\Gamma/H$, the rational Chern character is the identity
		under the canonical identifications with $R(H)_{\mathbb{Q}}$.
		Thus $e_{\Gamma/H}=\varepsilon_H$.
		If the $\Gamma$-action on $X$ is free, all stabilizers are
		trivial and $\varepsilon_{\{1\}}=1$, so $c_X=1$
		and hence $e_X=1$.
	\end{proof}

	\section{Proof of Theorem~\ref{thm:main}}

	Let $f:X\to X'$ be an equivariant continuous map between
	proper $\Gamma$-CW-complexes.
	For each finite $\Gamma$-subcomplex $K\subseteq X$,
	choose a finite $\Gamma$-subcomplex $K'\subseteq X'$
	containing $f(K)$.
	The restriction $f|_K:K\to K'$ is proper and therefore induces
	\[
	(f|_K)_*:KK_n^\Gamma(K,B)
	\longrightarrow KK_n^\Gamma(K',B).
	\]
	These homomorphisms are compatible with support inclusions
	and coefficient inclusions, and hence define
	\[
	f_*:KK_n^\Gamma(X,B)
	\longrightarrow KK_n^\Gamma(X',B)
	\]
	on the inductive limits.
	This construction is functorial.
	
	We apply this construction to the maps
	\[
	\begin{gathered}
		j:E\Gamma\longrightarrow E\Gamma\times\underline E\Gamma,
		\quad j(z)=(z,\pi(z)),
		\qquad
		q:E\Gamma\times\underline E\Gamma\longrightarrow E\Gamma,
		\quad q(z,x)=z,
		\\
		p:E\Gamma\times\underline E\Gamma\longrightarrow\underline E\Gamma,
		\quad p(z,x)=x.
	\end{gathered}
	\]
	Since $q\circ j=\operatorname{id}$, functoriality gives
	$q_*\circ j_*=\operatorname{id}$.
	Thus, for every $\Gamma$-$C^*$-algebra $A$, the homomorphism
	\[
	j_*:KK_*^\Gamma(E\Gamma,A)
	\longrightarrow
	KK_*^\Gamma(E\Gamma\times\underline E\Gamma,A)
	\]
	is split injective, with left inverse $q_*$.
	Since $\pi=p\circ j$, we also have
	$\pi_*=p_*\circ j_*$.
	It therefore suffices to prove that
	\[
	p_*:KK_*^\Gamma(E\Gamma\times\underline E\Gamma,A)_{\mathbb{Q}}
	\longrightarrow
	KK_*^\Gamma(\underline E\Gamma,A)_{\mathbb{Q}}
	\]
	is injective.
	Passing to inductive limits over finite $\Gamma$-subcomplexes
	of $\underline E\Gamma$ and separable $\Gamma$-invariant
	coefficient subalgebras, we reduce to proving that
	\[
	(p_X)_*:KK_*^\Gamma(E\Gamma\times X,B)_{\mathbb{Q}}
	\longrightarrow
	KK_*^\Gamma(X,B)_{\mathbb{Q}}
	\]
	is injective for every separable $\Gamma$-$C^*$-algebra $B$
	and every finite proper $\Gamma$-CW-complex $X$, where
	\[
	p_X:E\Gamma\times X\longrightarrow X,
	\qquad p_X(z,x)=x.
	\]
	
	We proceed in two steps.
	
	\paragraph{\bf The orbit case.}
	Set $E=E\Gamma$, and let $X=\Gamma/H$, where
	$H\leq\Gamma$ is a finite subgroup.

	It follows from \cite[Lemma~12.11]{GHT00}
	that the standard compression maps induce natural isomorphisms
	\[
	\operatorname{Comp}_T:
	KK_n^\Gamma(T\times X,B)_{\mathbb{Q}}
	\overset{\cong}{\longrightarrow}
	KK_n^H(T,B)_{\mathbb{Q}},
	\qquad T=E,~\mathrm{pt}.
	\]
	Here $\mathrm{pt}\times X=X$, and $B$ on the right carries
	the restricted $H$-action.
	
	Let
	\[
	r_T:KK_n^H(T,B)\longrightarrow KK_n(T,B),
	\qquad T=E,~\mathrm{pt}
	\]
	denote the forgetful maps.

	We next define the transfer map using the induction construction in
	\cite[Remark~3.6.8, Lemma~3.6.9 and
	Proposition~3.6.10]{Raven04}.
	For a finite $H$-subcomplex $\Delta\subseteq T$,
	Proposition~3.6.10 of \cite{Raven04} gives an isomorphism
	\begin{equation*}\label{eq:zhegeshiinductionmap}
		i_{\{1\}}^H:
		KK_n(C(\Delta),B)
		\overset{\cong}{\longrightarrow}
		KK_n^H\bigl(\operatorname{ind}_{\{1\}}^H C(\Delta),B\bigr).
	\end{equation*}
	Here
	$\operatorname{ind}_{\{1\}}^H C(\Delta)=C(H,C(\Delta))$,
	with $H$ acting by left translation:
	\[
	(\gamma\cdot f)(\gamma')=f(\gamma^{-1}\gamma'), \quad \gamma,\gamma'\in H.
	\]
	Define the $H$-equivariant homomorphism
	\[
	j_\Delta:C(\Delta)\longrightarrow
	\operatorname{ind}_{\{1\}}^H C(\Delta),
	\]
	by 
	\[
	j_\Delta(f)(\gamma)=\gamma^{-1}(f), \quad \gamma\in H.
	\]
	Define
	\[
	\operatorname{tr}_\Delta
	:=j_\Delta^*\circ i_{\{1\}}^H:
	KK_n(\Delta,B)\longrightarrow KK_n^H(\Delta,B).
	\]
	The naturality of $i_{\{1\}}^H$ and $j_\Delta$ makes these maps
	natural for proper equivariant maps and compatible with
	inclusions of finite supports.
	Passing to the inductive limits therefore gives the transfer maps
	\[
	\operatorname{tr}_T:
	KK_n(T,B)\longrightarrow KK_n^H(T,B),
	\qquad T=E,~\mathrm{pt}.
	\]

	After rationalization, define
	\[
	\begin{aligned}
		R_T:\quad
		KK_n^\Gamma(T\times X,B)_{\mathbb{Q}}
		&\xrightarrow{\operatorname{Comp}_T}
		KK_n^H(T,B)_{\mathbb{Q}}
		\xrightarrow{r_T}
		KK_n(T,B)_{\mathbb{Q}},\\
		J_T:\quad
		KK_n(T,B)_{\mathbb{Q}}
		&\xrightarrow{\operatorname{tr}_T}
		KK_n^H(T,B)_{\mathbb{Q}}
		\xrightarrow{\operatorname{Comp}_T^{-1}}
		KK_n^\Gamma(T\times X,B)_{\mathbb{Q}}.
	\end{aligned}
	\]
	
	We first compute the composites $J_T\circ R_T$.
	By definition,
	\[
	J_T\circ R_T
	=
	\operatorname{Comp}_T^{-1}
	\circ(\operatorname{tr}_T\circ r_T)
	\circ\operatorname{Comp}_T.
	\]
	Thus we begin with the composite
	$\operatorname{tr}_{\Delta}\circ r_{\Delta}$ for any finite $H$-subcomplex $\Delta$ of $T$,
	where $r_\Delta$ is the corresponding forgetful map
	on $\Delta$.
	
	In degree zero, by
	\cite[Proposition~20.2.4]{Blackadar},
	we may represent $x\in KK_0^H(\Delta,B)$ by an
	$H$-equivariant Kasparov module
	\[
	(\mathcal E,\varphi,F),
	\]
	where $\mathcal E$ is a countably generated graded Hilbert
	$B$-module equipped with a continuous, grading-preserving
	$H$-action,
	$\varphi:C(\Delta)\to\mathcal B(\mathcal E)$ is an
	$H$-equivariant graded $*$-homomorphism,
	and $F\in\mathcal B(\mathcal E)$ is a strictly
	$H$-invariant operator of degree one.
	For each $\gamma\in H$, let $\mathcal E_\gamma$ be a copy
	of the underlying graded vector space of $\mathcal E$,
	equipped with the Hilbert $B$-module structure
	\[
	\xi\cdot_\gamma b=\xi\,\gamma^{-1}(b),
	\qquad
	\langle\xi,\eta\rangle_\gamma
	=\gamma\bigl(\langle\xi,\eta\rangle\bigr).
	\]
	Define the representation $\varphi_\gamma$ 
	on this module by
	\[
	\varphi_\gamma(f)
	=
	\varphi\bigl(\gamma^{-1}(f)\bigr),
	\qquad f\in C(\Delta).
	\]
	Equip $\bigoplus_{\gamma\in H}\mathcal E_\gamma$
	with the $H$-action
	\[
	(\gamma\cdot\xi)_{\gamma'}
	=
	\xi_{\gamma^{-1}\gamma'},
	\qquad \gamma,\gamma'\in H, \quad \xi\in\bigoplus_{\gamma\in H}\mathcal E_\gamma.
	\]
	By the cycle description in
	\cite[Lemma~3.6.9]{Raven04}, followed by pullback along
	$j_\Delta$, the class
	$(\operatorname{tr}_\Delta\circ r_\Delta)(x)$
	is represented by
	\begin{equation}\label{eq:transfercocycleshibakwonimade}
		\left(
		\bigoplus_{\gamma\in H}\mathcal E_\gamma,\,
		\bigoplus_{\gamma\in H}\varphi_\gamma,\,
		\bigoplus_{\gamma\in H}F
		\right).
	\end{equation}

	Let $\{\delta_\gamma:\gamma\in H\}$ be the standard basis
	of $\mathbb{C}H$.
	The map
	\[
	U:\bigoplus_{\gamma\in H}\mathcal E_\gamma
	\longrightarrow\mathcal E\otimes\mathbb{C}H,
	\qquad
	U\big((\xi_\gamma)\big)
	=
	\sum_{\gamma\in H}(\gamma\cdot\xi_\gamma)\otimes\delta_\gamma,
	\]
	is a $B$-linear unitary intertwining the representations
	and operators.
	It is also $H$-equivariant for the diagonal action on
	$\mathcal E\otimes\mathbb{C}H$.
	Here $H$ acts on $\mathbb{C}H$ by the left regular
	representation.
	Consequently, the cycle \eqref{eq:transfercocycleshibakwonimade} is
	unitarily equivalent to
	\[
	\bigl(
	\mathcal E\otimes\mathbb{C}H,\,
	\varphi\otimes1,\,
	F\otimes1
	\bigr)
	\]
	with the diagonal $H$-action.
	Hence
	\begin{equation}\label{eq:zkkkkkkkkkkkkkkkkdkdkfkjd}
		(\operatorname{tr}_\Delta\circ r_\Delta)(x)
		=
		[\Delta\times\mathbb{C}H]\cdot x,
	\end{equation}
	where $[\Delta\times\mathbb{C}H]\in K_H^0(\Delta)$
	is the class of the trivial vector bundle $\Delta\times\mathbb{C}H$
	equipped
	with the diagonal $H$-action.
	Suspension gives the same identity in every degree.

	Let
	\[
	V=\Gamma\times_H\mathbb{C}H
	\]
	be the $\Gamma$-equivariant vector bundle over
	$X=\Gamma/H$, 
	with projection $[\gamma,v]\mapsto \gamma H$.
	Under the canonical identification
	$K_\Gamma^0(X)\cong R(H)$, the class $[V]$ therefore
	corresponds to $[\mathbb{C}H]$.
	By Lemma~\ref{lem:idempotent}$(1)$,
	\[
	[V]=|H|e_X
	\qquad\text{in }K_\Gamma^0(X)_{\mathbb{Q}}.
	\]
	
	For $T=\mathrm{pt}$, \eqref{eq:zkkkkkkkkkkkkkkkkdkdkfkjd}  gives
	\[
	\operatorname{tr}_{\mathrm{pt}}\circ r_{\mathrm{pt}}
	=
	[\mathrm{pt}\times\mathbb{C}H]\cdot-.
	\]
	Since the fiber of $V$ over $H\in X$ is $\mathbb{C}H$,
	the compression map is compatible with the corresponding
	vector-bundle actions: for every
	$y\in KK_n^\Gamma(X,B)_{\mathbb{Q}}$,
	\[
	\operatorname{Comp}_{\mathrm{pt}}([V]\cdot y)
	=
	[\mathrm{pt}\times\mathbb{C}H]\cdot
	\operatorname{Comp}_{\mathrm{pt}}(y).
	\]
	Consequently,
	\begin{equation}\label{eq:point-transfer-identity}
		\begin{aligned}
			J_{\mathrm{pt}}\circ R_{\mathrm{pt}}
			&=
			\operatorname{Comp}_{\mathrm{pt}}^{-1}
			\circ
			\bigl([\mathrm{pt}\times\mathbb{C}H]\cdot-\bigr)
			\circ\operatorname{Comp}_{\mathrm{pt}}\\
			&=[V]\cdot-\\
			&=|H|e_X\cdot-.
		\end{aligned}
	\end{equation}

	For $T=E$, 
	fix any finite $H$-subcomplex $\Delta$ of $E$. 
	Let
	\[
	Z_\Delta
	=
	\Gamma\cdot(\Delta\times\{H\})
	=
	\{(\gamma z,\gamma H):
	\gamma\in\Gamma,\ z\in\Delta\}
	\subseteq E\times X.
	\]
	There is a natural $\Gamma$-equivariant homeomorphism
	\[
	\Gamma\times_H\Delta
	\overset{\cong}{\longrightarrow}Z_\Delta,
	\qquad
	[\gamma,z]\longmapsto(\gamma z,\gamma H).
	\]
	The compression map gives an isomorphism
	\[
	\operatorname{Comp}_\Delta:
	KK_n^\Gamma(Z_\Delta,B)_{\mathbb{Q}}
	\overset{\cong}{\longrightarrow}
	KK_n^H(\Delta,B)_{\mathbb{Q}}.
	\]
	
	Let $\theta_\Delta:Z_\Delta\to X$ be the restriction of the
	second-coordinate projection $E\times X\to X$, given by
	\[
	\theta_\Delta(\gamma z,\gamma H)=\gamma H,
	\qquad \gamma\in\Gamma,\quad z\in\Delta.
	\]
	Then the pullback bundle $\theta_\Delta^*V$ has the form
	\[
	\theta_\Delta^*V
	\cong
	\Gamma\times_H(\Delta\times\mathbb{C}H),
	\]
	where $H$ acts diagonally on
	$\Delta\times\mathbb{C}H$.
	In particular, its restriction to
	$\Delta\times\{H\}$ is the trivial bundle
	$\Delta\times\mathbb{C}H$ with the diagonal $H$-action.
	As in the point case, the compression map is compatible with
	these bundle actions:
	\[
	\operatorname{Comp}_\Delta\big([\theta_\Delta^*V]\cdot y\big)
	=
	[\Delta\times\mathbb{C}H]\cdot
	\operatorname{Comp}_\Delta(y),
	\qquad
	y\in KK_n^\Gamma(Z_\Delta,B)_{\mathbb{Q}}.
	\]
	
	Since the $\Gamma$-action on $E$ is free, its action on
	$Z_\Delta\subseteq E\times X$ is also free.
	Using $[V]=|H|e_X$ and
	Lemma~\ref{lem:idempotent}(2)--(3), we obtain
	\[
	[\theta_\Delta^*V]
	=
	|H|\theta_\Delta^*e_X
	=
	|H|e_{Z_\Delta}
	=
	|H|\cdot1
	\qquad
	\text{in }K_\Gamma^0(Z_\Delta)_{\mathbb{Q}}.
	\]
	Therefore, for every
	$y\in KK_n^\Gamma(Z_\Delta,B)_{\mathbb{Q}}$, combining \eqref{eq:zkkkkkkkkkkkkkkkkdkdkfkjd}, we obtain
	\[
	\begin{aligned}
		&\bigl(
		\operatorname{Comp}_\Delta^{-1}
		\circ(\operatorname{tr}_\Delta\circ r_\Delta)
		\circ\operatorname{Comp}_\Delta
		\bigr)(y)\\
		&\qquad=
		\operatorname{Comp}_\Delta^{-1}
		\bigl(
		[\Delta\times\mathbb{C}H]\cdot
		\operatorname{Comp}_\Delta(y)
		\bigr)\\
		&\qquad=[\theta_\Delta^*V]\cdot y\\
		&\qquad=|H|\cdot y.
	\end{aligned}
	\]
	As $\Delta$ ranges over the finite $H$-subcomplexes of $E$,
	the subcomplexes $Z_\Delta$ are cofinal among the finite
	$\Gamma$-subcomplexes of $E\times X$.
	The compression, forgetful and transfer maps are compatible
	with the homomorphisms induced by support inclusions.
	Consequently,
	\begin{equation}\label{eq:wonisjdlfjldsjflkdskfj}
		J_E\circ R_E=|H|\cdot\operatorname{id}.
	\end{equation}
	
	We next show that 
	\begin{equation}\label{eq:finite-group-isomorphism}
		(p_X)_*:KK_n^\Gamma(E\times X,B)_{\mathbb{Q}}
		\overset{\cong}{\longrightarrow}
		e_XKK_n^\Gamma(X,B)_{\mathbb{Q}}.
	\end{equation}
	
	Since $E$ is contractible, homotopy invariance gives an
	isomorphism
	\[
	(p_{\mathrm{pt}})_*:KK_n(E,B)_{\mathbb{Q}}
	\overset{\cong}{\longrightarrow}
	KK_n(\mathrm{pt},B)_{\mathbb{Q}},
	\]
	where 
	$p_{\mathrm{pt}}:E\to\mathrm{pt}$ is the collapse map.
	The naturality of compression, forgetful, and transfer maps give
	\begin{equation}\label{eq:zhegshishemenewoyebuzhidaos}
		R_{\mathrm{pt}}\circ(p_X)_*=(p_{\mathrm{pt}})_*\circ R_E,
		\qquad
		(p_X)_*\circ J_E=J_{\mathrm{pt}}\circ (p_{\mathrm{pt}})_*.
	\end{equation}
	In particular, the first identity is expressed by the
	commutative diagram
	\[
	\begin{CD}
		KK_n^\Gamma(E\times X,B)_{\mathbb{Q}}
		@>{(p_X)_*}>>
		KK_n^\Gamma(X,B)_{\mathbb{Q}}\\
		@V{R_E}VV
		@VV{R_{\mathrm{pt}}}V\\
		KK_n(E,B)_{\mathbb{Q}}
		@>{(p_{\mathrm{pt}})_*}>>
		KK_n(\mathrm{pt},B)_{\mathbb{Q}}.
	\end{CD}
	\]
	
	Define
	\[
	\sigma=\frac{1}{|H|}
	J_E\circ (p_{\mathrm{pt}})_*^{-1}\circ R_{\mathrm{pt}}:
	KK_n^\Gamma(X,B)_{\mathbb{Q}}
	\longrightarrow
	KK_n^\Gamma(E\times X,B)_{\mathbb{Q}}.
	\]
	Using \eqref{eq:point-transfer-identity}, \eqref{eq:wonisjdlfjldsjflkdskfj} and \eqref{eq:zhegshishemenewoyebuzhidaos}, 
	we compute
	\[
	\begin{aligned}
		\sigma\circ(p_X)_*
		&=\frac{1}{|H|}
		J_E\circ (p_{\mathrm{pt}})_*^{-1}\circ R_{\mathrm{pt}}\circ(p_X)_*\\
		&=\frac{1}{|H|}
		J_E\circ (p_{\mathrm{pt}})_*^{-1}\circ (p_{\mathrm{pt}})_*\circ R_E\\
		&=\frac{1}{|H|}J_E\circ R_E\\
		&=\operatorname{id},
	\end{aligned}
	\]
	and
	\[
	\begin{aligned}
		(p_X)_*\circ\sigma
		&=\frac{1}{|H|}
		(p_X)_*\circ J_E\circ (p_{\mathrm{pt}})_*^{-1}\circ R_{\mathrm{pt}}\\
		&=\frac{1}{|H|}
		J_{\mathrm{pt}}\circ (p_{\mathrm{pt}})_*\circ (p_{\mathrm{pt}})_*^{-1}\circ R_{\mathrm{pt}}\\
		&=\frac{1}{|H|}
		J_{\mathrm{pt}}\circ R_{\mathrm{pt}}\\
		&=e_X\cdot-.
	\end{aligned}
	\]
	The first identity proves the injectivity of $(p_X)_*$.
	Moreover, for every
	$x\in KK_n^\Gamma(E\times X,B)_{\mathbb{Q}}$,
	\[
	e_X\cdot\big((p_X)_*(x)\big)
	=\bigl((p_X)_*\circ\sigma\circ(p_X)_*\bigr)(x)
	=(p_X)_*(x),
	\]
	so the image of $(p_X)_*$ lies in the $e_X$-summand.
	Conversely, if
	$y\in e_XKK_n^\Gamma(X,B)_{\mathbb{Q}}$, then
	\[
	\bigl((p_X)_*\circ\sigma\bigr)(y)=e_X\cdot y=y.
	\]
	We have therefore proved
	\eqref{eq:finite-group-isomorphism}.
	The inverse of \eqref{eq:finite-group-isomorphism} is the restriction of $\sigma$ to the summand
	on the right.

	\paragraph{\bf Finite proper $\Gamma$-CW-complex case.}
	We also set $E=E\Gamma$.
	For any  finite proper $\Gamma$-CW-complex $X$, we claim that
	\begin{equation}\label{eq:projection-isomorphism}
		(p_X)_*:KK_n^\Gamma(E\times X,B)_{\mathbb{Q}}
		\overset{\cong}{\longrightarrow}
		e_XKK_n^\Gamma(X,B)_{\mathbb{Q}}.
	\end{equation}
	
	For brevity, write
	\[
	F_n(W):=KK_n^\Gamma(E\times W,B)_{\mathbb{Q}},
	\qquad
	G_n(W):=e_WKK_n^\Gamma(W,B)_{\mathbb{Q}}
	\]
	for any finite proper $\Gamma$-CW-complex $W$.
	Let $x\in F_n(W)$ be represented by
	$x_L\in KK_n^\Gamma(L,B)_{\mathbb{Q}}$ on a finite
	$\Gamma$-subcomplex $L\subseteq E\times W$.
	The restricted projection $p_W|_L:L\to W$ is proper.
	Since $L$ is $\Gamma$-free, Lemma~\ref{lem:idempotent}(2)--(3) gives
	\[
	(p_W|_L)^*e_W=e_L=1.
	\]
	Proposition \ref{prop:action}$(1)$ therefore yields
	\[
	\begin{aligned}
		e_W\cdot\big((p_W)_*(x)\big)
		&=e_W\cdot\big((p_W|_L)_*(x_L)\big)\\
		&=(p_W|_L)_*\bigl((p_W|_L)^*e_W\cdot x_L\bigr)\\
		&=(p_W|_L)_*(x_L)\\
		&=(p_W)_*(x).
	\end{aligned}
	\]
	This  shows that $(p_W)_*$
	takes values in $G_n(W)$ for any finite proper $\Gamma$-CW-complex $W$.

	We proceed by induction on $\dim X$.
	If $\dim X=0$, then
	\[
	X=\bigsqcup_{\alpha=1}^{r}\Gamma/H_\alpha.
	\]
	The groups $F_n(X)$ and $G_n(X)$ decompose as direct sums
	over these orbits. Under these decompositions,
	\[
	(p_X)_*
	=
	\bigoplus_{\alpha=1}^{r}(p_{\Gamma/H_\alpha})_*.
	\]
	Each summand is an isomorphism by
	\eqref{eq:finite-group-isomorphism}, so $(p_X)_*$ is
	an isomorphism as well.

	Suppose that $m=\dim X\geq1$, and set $Z=X^{(m-1)}$.
	Write the $m$-cells as
	$\Gamma/H_\alpha\times D^m$, and set
	\[
	C=\bigsqcup_\alpha\Gamma/H_\alpha.
	\]
	Excision and suspension identify the relative terms for
	$(E\times X,E\times Z)$ and $(X,Z)$ with the corresponding
	groups of $C$, shifted by $m$.
	Thus the projections give a commutative diagram of exact sequences
	\begin{equation}\label{eq:cellular-induction-diagram}
		\begin{CD}
			F_{n+1-m}(C) @>>> F_n(Z) @>>> F_n(X)
			@>>> F_{n-m}(C) @>>> F_{n-1}(Z)\\
			@V{\cong}VV @V{\cong}VV @V{(p_X)_*}VV
			@V{\cong}VV @V{\cong}VV\\
			G_{n+1-m}(C) @>>> G_n(Z) @>>> G_n(X)
			@>>> G_{n-m}(C) @>>> G_{n-1}(Z).
		\end{CD}
	\end{equation}
	To verify exactness of the lower row,
	set
	\[
	M_k(W)=KK_k^\Gamma(W,B)_{\mathbb{Q}}.
	\]
	After excision and suspension, compatibility of the bundle
	actions with the long exact sequence for $(X,Z)$ gives
	the commutative diagram of exact sequences
	\[
	\begin{CD}
		M_{n+1-m}(C) @>>> M_n(Z) @>>> M_n(X)
		@>>> M_{n-m}(C) @>>> M_{n-1}(Z)\\
		@V{e_C\cdot-}VV @V{e_Z\cdot-}VV @V{e_X\cdot-}VV
		@V{e_C\cdot-}VV @V{e_Z\cdot-}VV\\
		M_{n+1-m}(C) @>>> M_n(Z) @>>> M_n(X)
		@>>> M_{n-m}(C) @>>> M_{n-1}(Z).
	\end{CD}
	\]
	Taking the images of these vertical idempotents gives
	the lower row of \eqref{eq:cellular-induction-diagram}, 
	which is
	therefore exact.
	
	The four outer vertical maps of \eqref{eq:cellular-induction-diagram} are isomorphisms by the orbit
	case and the induction hypothesis, so the Five Lemma proves
	\eqref{eq:projection-isomorphism}.

\end{document}